\documentclass[11pt]{article}
\usepackage[totalwidth=13.0cm,totalheight=20.0cm]{geometry}
\usepackage{latexsym,amsthm,amsmath,amssymb,url}

\usepackage[T1]{fontenc}
\usepackage{times}
\usepackage{soul}
\usepackage[hidelinks]{hyperref}
\usepackage[utf8]{inputenc}
\usepackage[small]{caption}
\usepackage{graphicx}
\usepackage{sgame}
\usepackage{amsthm}
\usepackage{enumitem}
\usepackage{booktabs}
\usepackage{algorithm}
\usepackage{algorithmic}
\usepackage[switch]{lineno}

\usepackage{amsfonts}
\usepackage{mathrsfs}
\usepackage{amssymb}
\usepackage{tikz}
\usetikzlibrary{calc,shapes}
\usetikzlibrary{snakes}
\usetikzlibrary{arrows}
\usetikzlibrary{arrows.meta}
\usepackage{multirow}
\usepackage{graphics}
\usepackage{xspace}
\usepackage{comment}

\newtheorem{theorem}{Theorem}[section]

\newtheorem{lemma}[theorem]{Lemma}

\newtheorem{conjecture}[theorem]{Conjecture}
\newtheorem{construction}[theorem]{Construction}

\newtheorem{claim}{Claim}

\newcommand{\GGr}[1]{{#1}}
\newcommand{\AYr}[1]{{#1}}

\newcommand{\QGr}[1]{{#1}}

\newcommand{\GG}[1]{{#1}}

\newcommand{\QG}[1]{{#1}}

\newcommand{\GGn}[1]{{#1}}

\newcommand{\YZn}[1]{{ #1}}
\newcommand{\QGn}[1]{{#1}}
\newcommand{\YLn}[1]{{#1}}

\newcommand{\wdom}[1]{\rightsquigarrow}

\title{Oriented discrepancy of Hamilton cycles in oriented graphs satisfying Ore-type condition}

\author{Jiangdong Ai\thanks{School of Mathematical Sciences and LPMC, Nankai University. {\tt jd@nankai.edu.cn}.}
\hspace{2mm}
Qiwen Guo\thanks{Department of Computer Science, Royal Holloway University of London. {\tt gqwmath@163.com}.}
\hspace{2mm}
Gregory Gutin\thanks{Department of Computer Science, Royal Holloway University of London, {\tt g.gutin@rhul.ac.uk}, and School of Mathematical Sciences and LPMC, Nankai University.}
\hspace{2mm} Yongxin Lan\thanks{School of Science, Hebei University of Technology, {\tt  2019110@hebut.edu.cn}.}
\\ Qi Shao\thanks{Center for Combinatorics and LPMC, Nankai University. {\tt sq1449682195@163.com}.}
\hspace{2mm} Anders Yeo\thanks{Department of Mathematics and Computer Science, University of Southern Denmark. {\tt andersyeo@gmail.com}, and Department of Mathematics, University of Johannesburg.}
\hspace{2mm} Yacong Zhou\thanks{Shenzhen Institutes of Advanced Technology, Chinese Academy of Sciences. {\tt yacong.zhou96@gmail.com}.}}

\begin{document}

\maketitle

\begin{abstract}
Erd{\H o}s (1963) initiated extensive graph discrepancy research on 2-edge-colored graphs. Gishboliner, Krivelevich, and Michaeli (2023) launched similar research on oriented graphs. They conjectured the following extension of Dirac's theorem: If $D$ is an oriented graph on $n \ge 3$ vertices with minimum degree $\delta (D) \ge n/ 2$, then $D$ contains a Hamilton oriented cycle with at least $\delta(D)$ arcs in the same direction. This conjecture was proved by Freschi and Lo (2024) who posed an open problem to extend their result to an Ore-type condition. We propose two conjectures for such extensions and prove results which provide support to the conjectures.
\end{abstract}

\section{Introduction}

For a hypergraph $\cal H$ and coloring $c: V({\cal H}) \rightarrow \{-1,1\}$, the discrepancy of an edge $e\in E({\cal H})$ is
$D_c(e)=|\sum_{v\in e} c(v)|.$ Combinatorial discrepancy studies $$\min_c \max_{e\in E({\cal H})}D_c(e),$$ the {\em discrepancy} of $\cal H$. For an excellent exposition of the topic, see e.g. \cite{M99}. Erd{\H o}s \cite{E63} launched an extensive study of graph discrepancy, which is a special case of combinatorial discrepancy: $V({\cal H})$ is the edge set of some graph $G$ and $E({\cal H})$ consists of the edge sets of specific subgraphs of $G.$ For recent papers on the topic, see e.g. \cite{BCJP,BCPT,FHLT,GKM22a,GKM22b}. 


Gishboliner, Krivelevich, and Michaeli \cite{GKM23} introduced \GGr{an} oriented version of graph discrepancy and conjectured the \GGr{statement} of Theorem~\ref{Dirac-type theorem}, stated below, \GGr{which was} proved by Freschi and Lo \cite{FL24}.

Let us now introduce some terminology and notation; terminology and notation not introduced in this paper can be found in \cite{BG00,BG18,BM08}.
The {\em underlying graph} of a digraph $D$ is an undirected \GG{multigraph $U(D)$ obtained from $D$ by removing the orientations of all arcs in $D$.} \QG{The {\em degree} of a vertex $x$ in a digraph $D$, denoted by $d_D(x)$, is the degree $d_{U(D)}(x)$ of $x$ in $U(D)$. }In what follows, we will often omit directed or undirected graph subscripts if such graphs are clear from the context. The \emph{minimum degree} of a vertex in a directed or undirected graph $H$ is denoted by $\delta(H)$. 

 The \emph{neighborhood} $N(x)$ of a vertex $x$ in a directed or undirected graph $D$ is the set of vertices adjacent to $v$. The \emph{closed neighborhood} of $v$ is $N[x]=N(x)\cup \{x\}$. \QGr{We use $K_n$ to denote the complete undirected graph of order $n$.  For an undirected graph $G$, let $\overline{G}$ denote its complement. }

A digraph $D$ is an {\em oriented graph} if there is no more than one arc between any pair of vertices in $D$. \GG{An {\em oriented cycle} ({\em oriented path}, respectively) in $D$ is a subdigraph $Q$ of $D$ such that $U(Q)$ is a cycle (path, respectively) in $U(D)$.}
 An oriented cycle (oriented path, respectively) in $D$ is \emph{Hamilton} if it is a spanning subdigraph of $D$. Note that a digraph $D$ has a Hamilton oriented cycle (path, respectively) if and only if $U(D)$ has a Hamilton cycle (path, respectively). 

For an oriented cycle  $C=x_1x_2\dots x_tx_1$, the arc $a$ between $x_i$ and $x_{i+1}$ (all subscripts are taken modulo $t$) on $C$ is \emph{forward} if $a=x_ix_{i+1}$ and  \emph{backward} if $a=x_{i+1}x_{i}$.  \GGr{(In some parts of the literature (e.g.,~\cite{GKM23}), an arc from $x$ to $y$ is written as $(x,y)$ and an edge between $x$ and $y$ as $\{x,y\}$. Following \cite{FL24}, we instead use the shorthand $xy$ for both $(x,y)$ and $\{x,y\}$ whenever the meaning is clear from context. This convention is also consistent with the standard notation for directed and undirected paths and cycles.)}
The number of forward (backward, respectively) arcs of $C$ is denoted by $\sigma^+(C)$ ($\sigma^-(C)$, respectively).
Let $\sigma_{\min}(C)=\min \left \{\sigma^+(C), \sigma^-(C)\right \}$ and $\sigma_{\max}(C)=\max \left \{\sigma^+(C),\sigma^-(C)\right \}$. Note that $\sigma_{\min}(C)$ and $\sigma_{\max}(C)$  do not depend on the order of the vertices of $C$. If all arcs of $C$ are forward, then $C$ is a {\em directed cycle}. Similar definitions and notation can be introduced for oriented paths.


\begin{theorem}[\cite{FL24}]\label{Dirac-type theorem}
Let $D$ be an oriented graph on $n\ge 3$ vertices. If $\delta(D) \ge \frac{n}{2}$, then there exists a Hamilton oriented cycle $C$ in $D$ such that $\sigma_{\max}(C)\ge \delta(D)$.
\end{theorem}
Note that Theorem  \ref{Dirac-type theorem}  is a strengthening of Dirac's theorem \cite{D52}: a graph $G$ on $n\ge 3$ vertices has a Hamilton cycle if $\delta(G)\ge n/2.$ Ore \cite{O60} extended Dirac's theorem as follows: a graph $G$ on $n\ge 3$ vertices has a Hamilton cycle, if $d(u)+d(v)\ge n$ for every pair of non-adjacent vertices $u$ and $v$  of $G$. In a digraph $D$, a pair of vertices are {\em non-adjacent} if there is no arc between them.


Freschi and Lo \cite{FL24} stated an open problem to extend their result to an Ore-type condition, \GGr{i.e.,  a lower bound on the sum of degrees of any pair of non-adjacent vertices.}
Now we pose two conjectures on the open problem. While we were unable to prove either conjecture, we show some results providing support to the conjectures.

By considering the minimum degree of $D$, we may extend Theorem \ref{Dirac-type theorem} to the following conjecture, which is an analogue of
Theorem \ref{Dirac-type theorem} with an Ore-type condition.




\begin{conjecture}\label{conj:1}
Let $D$ be an oriented graph on $n \ge 3$ vertices with minimum degree $\delta$. If $d(u)+d(v)\ge n$ for each pair of non-adjacent vertices $u$ and $v$, then there exists a Hamilton oriented cycle $C$ in $D$ such that $\sigma_{\max}(C) \ge \max\{\delta,n-\delta\}$.

\end{conjecture}
When $\delta\ge n/2$, Conjecture~\ref{conj:1} holds and is sharp, as it directly follows from Theorem~\ref{Dirac-type theorem}. The following construction shows that Conjecture~\ref{conj:1} is sharp when $\delta < n/2$, as $\sigma_{\max}(C) \leq  n-\delta$ for all Hamilton oriented cycles $C$ in the construction below.

\begin{construction}\label{construct:1}
Given integers $n$ and $\delta$ with $\delta\geq 2 $ and $n>2\delta$, let $G$ be a graph on $n$ vertices with vertex set $V(G)=A\cup B \cup \{v_0\}$, where $G[A]$ is isomorphic to $\overline{K}_{\delta}$, $G[B]$ is isomorphic to $K_{n-\delta-1}$, $N(v_0)=A $, every vertex in $A$ is adjacent to all vertices in $B$. Let $D$ be an orientation of $G$, such that every edge between $A$ and $B$ is oriented from $B$ to $A$, every edge between $v_0$ and $A$ is oriented from $v_0$ to $A$ (see Figure \ref{fig:teforc1}) and the orientation of every edge in $G[B]$ is arbitrary.
\end{construction}

\QGr{It is not hard to check that $d(u)+d(v)\ge n$ for each pair of non-adjacent vertices $u$ and $v$ in $D$ and $\delta(D)=\delta$.} Let $C$ be a Hamilton oriented cycle in $D$. Observe that every vertex in $A$ is incident to two edges of $C$ which are oriented in opposite directions. Since $A$ is an independent set, $\sigma_{\min}(C)\ge |A|=\delta$. It follows that $\sigma_{\max}(C)\le n-\delta$.

\begin{figure}[H]
\centering
		\tikzstyle{vertexY}=[circle,draw, top color=gray!5, bottom color=gray!30, minimum size=8pt, scale=0.6, inner sep=1pt]
		\begin{tikzpicture}[scale=1.5]
			
			\draw (2,0) ellipse (0.8 and 1);
			\draw (0,0) ellipse (0.5 and 0.6);
			\draw [thick, arrows = {-Stealth[reversed, reversed]}] (1.3,0.5)--(0.4,0.3);
			\draw [thick, arrows = {-Stealth[reversed, reversed]}] (1.3,-0.5)--(0.4,-0.3);
			\draw [thick, arrows = {-Stealth[reversed, reversed]}] (1.2,0)--(0.5,0);
			\draw (0,0) node {$\overline{K}_{\delta}$};
			\draw (2,0) node {$K_{n-\delta-1}$};
                \draw (0,-1.5) node {$A$};
                \draw (2,-1.5) node {$B$};
                \fill (-2,0) circle (.05);
                \draw (-2,0) node [left]{$v_0$};
                \draw [thick, arrows = {-Stealth[reversed, reversed]}] (-2,0)--(-0.5,0);
			\draw [thick, arrows = {-Stealth[reversed, reversed]}] (-2,0)--(-0.4,0.4);
			\draw [thick, arrows = {-Stealth[reversed, reversed]}] (-2,0)--(-0.4,-0.4);
			
		\end{tikzpicture}\hfill
        \caption{\QGn{Tight examples for} \YZn{Conjecture \ref{conj:1}}\label{fig:teforc1}}
\end{figure}
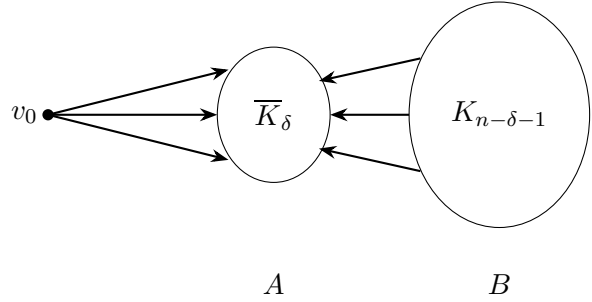

When we strengthen the degree sum condition in Conjecture \ref{conj:1}, we deduce the following result which achieves the desired bound stated in the conjecture.

	\QG{\begin{theorem}\label{thm: degree sum}
		Let $D$ be an oriented graph on $n$ vertices with minimum degree $1<\delta<\frac{n}{2}$. If $d(u)+d(v)\ge n+\delta-2$ for each pair of non-adjacent vertices $u$ and $v$, then there exists a Hamilton oriented cycle $C$ in $D$ such that $\sigma_{\max}(C) \ge n-\delta$.
	\end{theorem}}

Observe that Theorem \ref{thm: degree sum} implies that Conjecture \ref{conj:1} is true for $\delta=2$, \QGr{since in this case the lower bound of degree sum of the condition in  Theorem \ref{thm: degree sum} is precisely $n$.}

When \YLn{the subgraph induced by} the neighborhood of a vertex with the minimum degree in the graph is isomorphic to a tournament, we can derive a slightly weaker bound than in Conjecture \ref{conj:1}.

\begin{theorem}\label{clique-N(v)}
    Let $D$ be an oriented graph on $n$ vertices with minimum degree $3 \le \delta \le \frac{n}{3}$, which satisfies \QGr{that $d(u)+d(v)\ge n$ for each pair of non-adjacent vertices $u$ and $v$ in $D$}. If there is a vertex $v_0 \in V(D)$ such that $d(v_0)=\delta$ and $D[N(v_0)]$ is isomorphic to a tournament, then there is a Hamilton oriented cycle $C$ in $D$ such that \QG{$\sigma_{\max}(C) \ge n-\delta-2$}.
\end{theorem}


For every $n$-vertex oriented graph $D$, we define $$s^*(D)=\min\{d(u)+d(v)-n: u\ne v\in V(D), \{uv,vu\} \cap A(D)=\emptyset\}$$
if there is a pair of non-adjacent vertices in $D$, and $s^*(D)=n-2$ if $D$ is a tournament. The following conjecture is another possible generalization of Theorem \ref{Dirac-type theorem}. 

	\begin{conjecture}\label{conj:2}
		Let $D$ be an oriented graph on \GGn{$n\ge 3$} vertices. If $s^*(D)\geq 0$, then there is a Hamilton oriented cycle $C$ in $D$ such that $\sigma_{\max}(C)\geq \lceil\frac{n+s^*(D)}{2}\rceil$.
	\end{conjecture}

	If Conjecture \ref{conj:2} is true, it also implies Theorem \ref{Dirac-type theorem}. Indeed, when $\delta(D) \ge n/2$, since $s^*(D)\geq 2\delta(D)-n$, Conjecture \ref{conj:2} implies that $\sigma_{\max}(C)\geq n/2 + (2\delta(D)-n)/2=\delta(D)$. 

     The following \YLn{construction shows} the bound in Conjecture \ref{conj:2} is tight for every positive integer $n$ and integer $0\le s^*(D) \le n-4$ (note that if $ s^*(D)\ge n-3$ then we are considering tournaments and the bounds are clearly tight for them).
   \QG{\begin{construction}
        Given a positive integer $n$ and a non-negative integer $k$ with  $n\ge k+4$, let \QGr{$c \in \{0, 1\}$ with $c\equiv n+k \pmod{2}$} . Let $G$ be a graph on $n$ vertices with vertex set $V(G)=A\cup B$, where $G[A]$ is isomorphic to \QGr{$\overline{K}_{\frac{n-k-c}{2}}$}, $G[B]$ is isomorphic to \QGr{$K_{\frac{n+k+c}{2}} $}, and all possible edges with one end-vertex in $A$ and the other in $B$ exist except \QGr{one edge if $c=1$.} Let $D$ be an orientation of $G$, such that every edge between $A$ and $B$ is oriented from $B$ to $A$, and the orientation of every edge in $G[B]$ is arbitrary (see Figure \ref{fig:teforc2}).
   \end{construction}}

        It is not hard to check that $s^*(D)=k$, and for every Hamilton oriented cycle $C$ in $D$, \QGr{every vertex in $A$ is incident to two edges of $C$ which are oriented in opposite directions, and since $A$ is an independent set,} we have $\sigma_{\max}(C)\leq \lceil\frac{n+k}{2}\rceil$.
	
	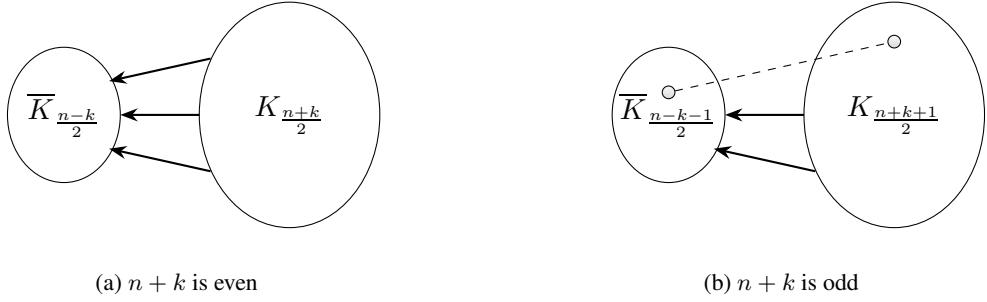
\begin{figure}[H]
		\tikzstyle{vertexY}=[circle,draw, top color=gray!5, bottom color=gray!30, minimum size=8pt, scale=0.6, inner sep=1pt]
		\begin{tikzpicture}[scale=1.5]
			
			\draw (2,0) ellipse (0.8 and 1);
			\draw (0,0) ellipse (0.5 and 0.6);
			\draw [thick, arrows = {-Stealth[reversed, reversed]}] (1.3,0.5)--(0.4,0.3);
			\draw [thick, arrows = {-Stealth[reversed, reversed]}] (1.3,-0.5)--(0.4,-0.3);
			\draw [thick, arrows = {-Stealth[reversed, reversed]}] (1.2,0)--(0.5,0);
			\draw (0,0) node {$\overline{K}_{\frac{n-k}{2}}$};
			\draw (2,0) node {$K_{\frac{n+k}{2}}$};
			\draw (1,-1.5) node {{\footnotesize (a) $n+k$ is even}};
			
		\end{tikzpicture}\hfill
		\begin{tikzpicture}[scale=1.5]
				\node (v1) at (0,0.2) [vertexY] {};
					\node (v2) at (2,0.65) [vertexY] {};
			\draw (2,0) ellipse (0.8 and 1);
			\draw (0,0) ellipse (0.5 and 0.6);
		  	  \draw [dashed] (v1)--(v2);
		\draw [thick, arrows = {-Stealth[reversed, reversed]}] (1.3,-0.5)--(0.4,-0.3);
		\draw [thick, arrows = {-Stealth[reversed, reversed]}] (1.2,0)--(0.5,0);
			\draw (0,0) node {$\overline{K}_{\frac{n-k-1}{2}}$};
			\draw (2,0) node {$K_{\frac{n+k+1}{2}}$};
			\draw (1,-1.5) node {{\footnotesize (b) $n+k$ is odd}};
		\end{tikzpicture}
		\centering
		\caption{Tight examples for Conjecture \ref{conj:2}}\label{fig:teforc2}
	\end{figure}

By applying an approach from \cite{GKM23}, we can obtain the following approximate version of Conjecture \ref{conj:2}.
	
	\begin{theorem}\label{thm:approx}
		For any integer $k\geq 0$ and oriented graph $D$ with order $n\geq 30+4(k-1)$. If $s^*(D)\geq 8k$, then there is a Hamilton oriented cycle $C$ in $D$ such that $\sigma_{\max}(C)\geq \lceil\frac{n+k}{2}\rceil$.
	\end{theorem}

The rest of this paper is organized as follows: in Section \ref{sec2}, we prove Theorem \ref{thm: degree sum}, in Section \ref{sec:clique-N(v)}, we prove Theorem \ref{clique-N(v)}, and in Section \ref{sec4}, we prove Theorem \ref{thm:approx}. 
\GGr{We conclude the paper with Section \ref{sec:further}, where we briefly discuss further research of digraph discrepancy.}





\section{Proof of Theorem \ref{thm: degree sum}}\label{sec2}
\QGr{The core technique of proving Theorem \ref{thm: degree sum} is to extend an oriented path or cycle by adding new vertices to it. This extension preserves the order of its original vertices with its discrepancy not decreasing. Specifically, We first identify a large oriented cycle $C$ with high discrepancy in the digraph $D$. Then we add all the vertices in $V(D) \setminus V(C)$ to $C$ to extend it to a Hamilton oriented cycle $C'$ while preserving high discrepancy. 

To accomplish this, Lemma \ref{lem:add x between pair} is needed. It allows us to add a sufficient number of vertices to an oriented path or cycle while not decreasing its discrepancy.}

	\begin{lemma}\label{lem:add x between pair}

    \QGr{Let $D$ be an oriented graph which contains an oriented path $P$ and a set of vertices $K$ disjoint from $V(P)$, such that $V(P)\subseteq N(v)$ for each $v\in K$. If $|V(P)|\ge |K|+2$, then $D$ contains an oriented path $P'$ with $V(P')=V(P)\cup K$ and same initial and terminal vertices as $P$, satisfying $\sigma^+(P')\ge \sigma^+(P) $. Similarly, an oriented path $P''$ satisfying $\sigma^-(P'') \ge \sigma^-(P)$ also exists.}
	\end{lemma}

    \begin{proof}
        \AYr{  Let $K=\{v_1,v_2,\dots,v_k\}$ and  $P = u_1u_2\ldots u_{l}$, where $l \geq k+2$. Let $A=\{a_1,a_2,\ldots,a_k\}$ be $k$ distinct arcs chosen randomly  and uniformly from $A(P)$. For each $i \in [k]$ let $a_i=s_it_i$ and let $P'$ be obtained from the path $P$,
        by substituting $a_i$ with the oriented path $s_i v_i t_i$ for each $i \in [k]$.
        
          Let $q$ and $r$ be any integers such that $2 \leq q \leq l-1$ and $r \in [k]$.
          If  $u_q v_r \in A(D)$ then $u_qv_r$ will be a forward arc in $P'$ with probability $\frac{1}{l-1}$, as this is the case if and only if $a_r = u_q u_{q+1}$.
        And, if  $v_r u_q \in A(D)$ then $v_r u_q$ will be a forward arc in $P'$ with probability $\frac{1}{l-1}$, as
         this is the case if $a_r = u_{q-1} u_{q}$. So every arc between $K$ and $\{u_2,u_3,u_4, \ldots , u_{l-1}\}$ will be a forward arc
        with probability $\frac{1}{l-1}$ and there are $k \cdot (l-2)$ such arcs.  And we remove at most $k$ forward arcs from $P$ when we remove the arcs in $A$. So, on average, $P'$
        will satisfy the following (as $l \geq k+2$),
         \[
        \sigma^+(P') \geq \sigma^+(P) + \frac{k\cdot (l-2)}{l-1} - k \geq \sigma^+(P) + \frac{k\cdot k}{k+1} - k >  \sigma^+(P) -1
        \]
        As $\sigma^+(P')$ is an integer this implies that there exists an oriented path $P'$ with $\sigma^+(P') \geq \sigma^+(P)$.
        Proving that there exists an oriented path $P''$ satisfying $\sigma^-(P'') \ge \sigma^-(P)$ can be done analogously.}
	\end{proof}

\GGr{Clearly, Lemma~\ref{lem:add x between pair} is of interest in itself in study of discrepancy of oriented paths and cycles. In algorithmic applications, naturally we may ask whether 
the proof of Lemma~\ref{lem:add x between pair} can be derandomized leading to a polynomial-time algorithm. Consider the complete bipartite graph $B$ with partite sets $A(P)$ and $K$. For an arc $a=xy\in A(P)$ and a vertex $v\in K$, let $P(a,v)$ be the path obtained from $P$ by replacing $a$ with the path $xvy$ and let the weight $w(av)$ be equal to $\sigma^+(P(a,v))-\sigma^+(P)$. 
Note that the set $A=\{a_1,a_2,\ldots,a_k\}$ of chosen arcs of $P$ can be viewed of as a matching $\{a_i v_i \; | \; i \in [k] \}$ in $B$. }
\GGr{The proof of Lemma~\ref{lem:add x between pair} essentially shows that the average weight of a matching in $B$ that saturates $K$ is strictly greater than $-1$.
As we, in polynomial time, can find a maximum weight matching that saturates $K$ whose weight is at least the average weight of such a matching,
we can, in polynomial time, create a path $P'$ with $\sigma^+(P') \geq \sigma^+(P)$ (by substituting each $a_i$ by $s_i v_i t_i$ for each edge $a_i v_i$ in the maximum weight matching).
}

\vspace{2mm}
	
    \QG{With Lemma \ref{lem:add x between pair} established, we are now ready to prove Theorem \ref{thm: degree sum}.}\\
	
    \QG{\noindent\textbf{Theorem \ref{thm: degree sum}.} \textit{Let $D$ be an oriented graph on $n$ vertices with minimum degree $1<\delta<\frac{n}{2}$. If $d(u)+d(v)\ge n+\delta-2$ for each pair of non-adjacent vertices $u$ and $v$, then there exists a Hamilton oriented cycle $C$ in $D$ such that $\sigma_{\max}(C) \ge n-\delta$.}}

	\begin{proof}
    Let $v_0\in V(D)$ and $d(v_0)=\delta$. Let $N(v_0)=\{v_1,v_2,\ldots,v_{\delta}\}$ and $U=V(D)\setminus N[v_0]=\{u_1,u_2,\ldots, u_{n-\delta-1}\}$. Since $d(u)+d(v)\ge n+\delta-2$ for any non-adjacent $u,v \in V(D)$, $d(u)\ge n-2$ for all $u\in U$, which means $u$ is adjacent to all vertices except $v_0$ in $D$.
		
		For $n=5$, there is a Hamilton oriented cycle $C$ in $D$ such that $\sigma_{\max}(C)\ge 3\ge n-\delta$. For $n\ge 6$, since  $D[U]$ is isomorphic to a tournament of order $n-\delta -1$, Theorem \ref{Dirac-type theorem} implies that there is a Hamilton oriented cycle $C_0$ in $D[U]$ such that $\sigma_{\max}(C_0)\ge n-\delta -2$. Without loss of generality, we assume that $C_0 = u_1u_2\ldots u_{n-\delta -1}u_1$, $\sigma_{\max}(C_0) = \sigma^{+}(C_0)$ and $u_2u_1$ is the only backward arc in $C_0$ (If $\sigma_{\max}(C_0)= n-\delta -2$). For each $i\in [n-\delta -1]$, we get Hamilton oriented cycles $C^1_{i,i+1}=u_iv_1v_0v_2u_{i+1}u_{i+2}\dots u_{i-1}u_i$ and $C^2_{i,i+1}=u_iv_2v_0v_1u_{i+1}u_{i+2}\dots u_{i-1}u_i$ (all subscripts are taken modulo $n-\delta -1$)  in $D[U\cup \{v_0,v_1,v_2\}]$.
		\begin{claim}\label{claim_1}
			\YLn{There exists $C^t_{i,i+1}$ such that $\sigma^+(C^t_{i,i+1})\ge n-\delta$ for some $i\in [n-\delta -1]$ and $t\in [2]$}.
		\end{claim}
		Suppose that $\sigma^+(C^t_{i,i+1}) \le n-\delta -1$ for each $i\in [n-\delta -1]$ and $t\in [2]$. We may assume $\{v_1v_0,v_2v_0\}\subset A(D)$ or $\{v_0v_1,v_0v_2\}\subset A(D)$, since otherwise $\sigma^+(C^1_{1,2}) \ge n-\delta$  if $v_1v_0v_2$ is a directed path or $\sigma^+(C^2_{1,2}) \ge n-\delta$ if $v_2v_0v_1$ is a directed path. Without loss of generality, assume $\{v_1v_0,v_2v_0\}\subset A(D)$. Since $v_1v_0\in A(D)$, $ \sigma^+(C^1_{1,2}) \ge n-\delta -1$ and $ \sigma^+(C^1_{1,2}) = n-\delta -1$. Thus $\{v_1u_1, u_2v_2\}\subset A(D)$. Similarly, since $v_2v_0\in A(D)$, $\sigma^+(C^2_{1,2}) = n-\delta -1$ and $\{v_2u_1, u_2v_1\}\subset A(D)$. \QGr{Note that, for any $i\in [n-\delta -1]\setminus\{1\}$, if $\{u_iv_1, v_1v_0\} \subset A(D)$, we have $ \sigma^+(C^1_{i,i+1}) = n-\delta -1$ and then $u_{i+1}v_2\in A(D)$, and  if  $\{u_iv_2, v_2v_0\} \subset A(D)$, we have $ \sigma^+(C^2_{i,i+1}) = n-\delta -1$ and then $u_{i+1}v_1\in A(D)$. Now, since $\{v_1v_0,v_2v_0\}\subset A(D)$ and $\{u_2v_1, u_2v_2\}\subset A(D)$,} we have $\{u_{i+1}v_1,u_{i+1}v_2\}\subset A(D)$ for every $i\in [n-\delta -1]$. Thus, $\{u_1v_1,u_1v_2,v_1u_1, v_2u_1\}\subset A(D)$, which contradicts the condition that $D$ is an oriented graph. Thus Claim \ref{claim_1} holds.
		\vspace{0.5cm}
		
        Let $C^{t_0}_{i_0,i_0+1}$ be the oriented cycle such that $\sigma^+(C^{t_0}_{i_0,i_0+1})\ge n-\delta$ as in Claim \ref{claim_1}. Consider the oriented path $P=u_{i_0+1}u_{i_0+2}\dots u_{i_0}$ (all subscripts are taken modulo $n-\delta-1$). Since $\delta-2 \le (n-\delta -1) -2$, Lemma \ref{lem:add x between pair} implies that there exists an oriented path $P'$ with $V(P')=V(P)\cup \{v_3,\dots,v_{\delta}\}$ and the same initial and terminal vertices as $P$, satisfying $\sigma^+(P') \ge \sigma^+(P)$. Let $C$ be the Hamilton oriented cycle in $D$, whose arc set consists of $A(C^{t_0}_{i_0,i_0+1})\setminus A(P)$ and $A(P')$. Thus, $\sigma_{\max}(C)\ge \sigma^+(C^{t_0}_{i_0,i_0+1})=n-\delta$.
        \end{proof}

\section{Proof of Theorem \ref{clique-N(v)}}\label{sec:clique-N(v)}
  Firstly, we need the following lemma.

    \begin{lemma}\label{exchange vertices}
        Let $D$ be a tournament on $n\ge 3$ vertices. For a Hamilton oriented cycle $C=v_1 v_2 \dots v_{i-1}v_i v_{i+1}\dots v_{j-1}v_j v_{j+1}\dots v_n v_1$ in $D$, \QG{let $C'$  be a Hamilton oriented cycle which is obtained from $C$, satisfying $C'=v_1 v_2\dots v_{i-1}v_j v_{i+1}\dots v_{j-1} v_i v_{j+1}\dots\\ v_n v_1$.} 
        Then $\sigma^+(C') \ge \sigma^+(C)-4$ and $\sigma^-(C') \ge \sigma^-(C)-4$. In particular, if $v_i$ and $v_j$ are adjacent in $C$, then $\sigma^+(C') \ge \sigma^+(C) - 3$ and $\sigma^-(C') \ge \sigma^-(C) - 3$.
    \end{lemma}

    \begin{proof}
       \QG{$C'$ is definitely an oriented cycle, as $D$ is a tournament. Since all vertices except $v_i$ and $v_j$ appear in the same order in $C$ and $C'$ and there are at most four arcs in $C$ incident with $v_i$ or $v_j$, we have that $\sigma^+(C') \ge \sigma^+(C)-4$ and $\sigma^-(C') \ge \sigma^-(C)-4$.}
    \end{proof}

   \QG{\noindent\textbf{Theorem \ref{clique-N(v)}.} \textit{Let $D$ be an oriented graph on $n$ vertices with minimum degree $3 \le \delta \le \frac{n}{3}$, which satisfies \QGr{that $d(u)+d(v)\ge n$ for each pair of non-adjacent vertices $u$ and $v$ in $D$}. If there is a vertex $v_0 \in V(D)$ such that $d(v_0)=\delta$ and $D[N(v_0)]$ is isomorphic to a tournament, then there is a Hamilton oriented cycle $C$ in $D$ such that $\sigma_{\max}(C) \ge n-\delta-2$.}}

  \begin{proof}Suppose that $v_0 \in V(D)$ with $d(v_0) = \delta$ and $D[N(v_0)]$ is isomorphic to a tournament. We denote $V(D)\backslash N[v_0]$ by $W$. Then $|W|=n-\delta-1$. Since for every vertex $w \in W$, we have $d(w) \ge n-\delta$ and $n-2\delta \le d_{D[W]}(w) \le n-\delta-2$, then $w$ must be adjacent to at least two vertices in $N(v_0)$.

  Since  $\delta(D[W]) \ge n-2\delta \ge \frac{n-\delta-1}{2}= \frac{|W|}{2}$ when $\delta \le \frac{n}{3}$, Theorem~\ref{Dirac-type theorem} implies that there is a Hamilton oriented cycle $C_1$ in $D[W]$ such that $\sigma_{\max}(C_1)\ge \delta(D[W])$. Let $C_1=w_1w_2 \ldots w_{n-\delta-1}w_1$. Similarly, there exists an Hamilton oriented cycle $C_2 = v_0 v_1 v_2\ldots v_{\delta}v_0$ with $\sigma_{\max}(C_2) \ge \delta$  in $D[N[v_0]]$. 
  Let $w_k \in W$ such that $d_{D[W]}(w_k) = \delta(D[W])$. 
  Since $V(C_1)\cap V(C_2)=\emptyset$, without loss of generality, we may assume that $\sigma^+(C_i) = \sigma_{\max}(C_i)$, for each $i \in [2]$.

\begin{claim}\label{claim:2.2-1}
     If there exist $v_{j}$ and $v_{j+1}$ for $j\in [\delta-1]$ in $C_2$ that are adjacent to $w_{k+1}$ and $w_{k}$ respectively (all subscripts are taken modulo $n-\delta-1$), then there is a Hamilton oriented cycle $C$ in $D$ such that $\sigma_{\max}(C) \ge n-\delta-2$.
\end{claim}
 Consider $C = w_1\ldots w_k v_{j+1}v_{j+2} \ldots v_{\delta} v_0 v_1\dots v_{j} w_{k+1} \ldots w_{n-\delta-1}w_1$, which is obtained from $C_1$ by deleting the arc between $w_k$ and $w_{k+1}$, and adding the arc between $w_k$ and $v_{j+1}$, the arc between $w_{k+1}$ and $v_{j}$, and the longer oriented path on $C_2$ from $v_{j+1}$ to $v_{j}$. We have $\sigma_{\max}(C) \ge \sigma^+(C)\ge n-2\delta - 1 + \delta -1=n-\delta -2$. Thus Claim \ref{claim:2.2-1} holds.\\

Now, we may assume that for each $j\in [\delta-1]$, if $w_{k+1}$ is adjacent to $v_j$, then $w_{k}$ is not adjacent to $v_{j+1}$ in the following discussion. Then, $w_k$ is not adjacent to at least one vertex in $V(C_2)\setminus\{v_0\}$ since $w_{k+1}$ has at least two neighbors in $V(C_2)\setminus\{v_0\}$. In this case, $\delta(D[W])\ge n-2\delta+1$.
\begin{claim}\label{claim:2.2-2}
    If $v_1$ and $v_{\delta}$ are adjacent to $w_k$ and $w_{k+1} \pmod{n-\delta-1}$ respectively, then there is a Hamilton oriented cycle $C$ in $D$ such that $\sigma_{\max}(C) \ge n-\delta-2$.
\end{claim}
 Consider $C' = w_1\ldots w_k v_{1}v_{2} \ldots v_{\delta}w_{k+1} \ldots w_{n-\delta-1}w_1$, which is obtained from $C_1$ by deleting the arc between $w_k$ and $w_{k+1}$, and adding the arc between $w_k$ and $v_{1}$, the arc between $w_{k+1}$ and $v_{\delta}$ and \QGr{the oriented path $C_2-\{v_0\}$.} We have $\sigma^+(C') \ge n-2\delta +1 - 1 + \delta -2=n-\delta -2$. By Lemma \ref{lem:add x between pair}, we can add $v_0$ to $C'$ when $\delta \ge 3$, to obtain the desired Hamilton oriented cycle $C$ in $D$ such that $\sigma_{\max}(C)\ge  \sigma_{\max}(C') \ge n-\delta - 2$.  Thus Claim \ref{claim:2.2-2} holds.\\

 If $\delta(D[W])=n-2\delta+1$, then $w_k$ is not adjacent to at most one vertex in $V(C_2)\setminus\{v_0\}$. By the assumption before, $w_k$ is not adjacent to exactly one vertex $v_{i_1}$ in $V(C_2)\setminus\{v_0\}$. Since $w_{k+1}$ is adjacent to at least two vertices in $V(C_2)\setminus\{v_0\}$, we may assume that $v_{j_1}$ and $v_{j_2}$ are adjacent to $w_{k+1}$ and $j_1<j_2$. By the assumption before, we have $i_1=j_1+1$ and $j_2=\delta$. Then $v_1$ and $v_{\delta}$ are adjacent to $w_k$ and $w_{k+1}$ respectively. By Claim \ref{claim:2.2-2}, we are done.

 We may also assume that either $w_{k}$ is not adjacent to $v_{1}$ or $w_{k+1}$ is not adjacent to $v_{\delta}$ in the following discussion. In this case, $\delta(D[W])\ge n-2\delta+2$.

 \begin{claim}\label{claim:2.2-3}
     If there exist $v_{j}$ and $v_{j+1}$ for $j\in [\delta-1]$ in $C_2$ that are adjacent to $w_{k}$ and $w_{k+1}$ (all subscripts are taken modulo $n-\delta-1$) respectively, then there is a Hamilton oriented cycle $C$ in $D$ such that $\sigma_{\max}(C) \ge n-\delta-2$.
 \end{claim}
   Let $C'_2$ be an oriented cycle which is obtained from $C_2$ in Lemma \ref{exchange vertices}, by changing $v_{j}$ and $v_{j+1}$ in $C_2$. By Lemma~\ref{exchange vertices}, $\sigma^+(C_2')\ge \sigma^+(C_2)-3\ge \delta - 3$. Consider $C = w_1\ldots w_k v_{j}v_{j+2} \ldots v_{\delta} v_0 v_1\dots v_{j-1}v_{j+1} w_{k+1} \ldots w_{n-\delta-1}w_1$, which is obtained from $C_1$ by deleting the arc between $w_k$ and $w_{k+1}$, and adding the arc between $w_k$ and $v_{j}$, the arc between $w_{k+1}$ and $v_{j+1}$, and the longer path on $C'_2$ from $v_{j}$ to $v_{j+1}$. We have $\sigma_{\max}(C) \ge \sigma^+(C)\ge n-2\delta +2 - 1 + \delta -3=n-\delta -2$. Thus Claim \ref{claim:2.2-3} holds.\\

   Based on the former two assumptions, we may assume that for each $j\in [\delta-1]$, if $w_{k+1}$ is adjacent to $v_{j+1}$, then $w_{k}$ is not adjacent to $v_j$ in the following discussion.

   If $\delta(D[W])= n-2\delta+2$, then $w_k$ is not adjacent to at most two vertices in $V(C_2)\setminus\{v_0\}$. By the assumptions before, $w_k$ is not adjacent to exactly two vertices $v_{i_1}$ and $v_{i_2}$ in $V(C_2)\setminus\{v_0\}$, in which $i_1<i_2$. Since $w_{k+1}$ is adjacent to at least two vertices in $V(C_2)\setminus\{v_0\}$, we may assume that $v_{j_1}$ and $v_{j_2}$ are adjacent to $w_{k+1}$ and $j_1<j_2$. By the assumptions before, $i_1=j_1+1=j_2-1$ and $i_2=j_2+1$ and $j_1=1$, or $i_1=j_1-1$ and $i_2=j_1+1=j_2-1$ and $j_2=\delta$. Then both of $w_k$ and $w_{k+1}$ are adjacent to $v_1$ and $v_3$, or $v_{\delta-2}$ and $v_{\delta}$. \QGr{If both of $w_k$ and $w_{k+1}$ are adjacent to $v_1$ and $v_3$, let $C'_2$ be an oriented cycle which is obtained from $C_2$ in Lemma \ref{exchange vertices} by changing $v_2$ and $v_3$ in $C_2$, and} consider $C = w_1\ldots w_k v_3v_2v_4 \ldots v_{\delta} v_0 v_1w_{k+1} \ldots w_{n-\delta-1}w_1$, which is obatined from $C_1$ by deleting the arc between $w_k$ and $w_{k+1}$, and adding the arc between $w_k$ and $v_3$, the arc between $w_{k+1}$ and $v_1$, and the longer oriented path of $C'_2$ from $v_3$ to $v_1$. \QGr{If not, then both of $w_k$ and $w_{k+1}$ are adjacent to $v_{\delta-2}$ and $v_{\delta}$, let $C'_2$ be an oriented cycle which is obtained from $C_2$ in Lemma \ref{exchange vertices} by changing $v_{\delta-2}$ and $v_{\delta-1}$ in $C_2$, and} consider $C = w_1\ldots w_k v_{\delta} v_0 v_1 \ldots v_{\delta-3}v_{\delta-1}v_{\delta-2} w_{k+1} \ldots w_{n-\delta-1}w_1$, which is obtained from $C_1$ by deleting the arc between $w_k$ and $w_{k+1}$, and adding the arc between $w_k$ and $v_{\delta}$, the arc between $w_{k+1}$ and $v_{\delta-2}$, and the longer oriented path of $C'_2$ from $v_{\delta}$ to $v_{\delta-2}$. \QGr{For both cases, we have $\sigma^+(C_2')\ge \sigma^+(C_2)-3\ge \delta - 3$ by Lemma \ref{exchange vertices}, and then} $\sigma_{\max}(C) \ge \sigma^+(C)\ge n-2\delta +2 - 1 + \delta -3=n-\delta -2$.

   If $\delta(D[W])\ge n-2\delta+3$, then assume that the vertices in $V(C_2)$ which are adjacent to $w_{k+1}$ are $\{v_{j_1},\ldots,v_{j_a}\}$, in which $a \ge 2$ and $j_1<j_2<\dots<j_a$. Then $w_k$ is not adjacent to $v_{j_1+1}$ by the assumptions before. Let $v_{i_1}$ be a vertex in $V(C_2)\setminus \{v_0\}$ that is adjacent to $w_k$. Without loss of generality, assume $i_1< j_1-1$. Let $C'_2$ be an oriented cycle which is obtained from $C_2$ in Lemma \ref{exchange vertices}, by changing $v_{i_1}$ and $v_{j_1+1}$ in $C_2$. By Lemma~\ref{exchange vertices}, $\sigma^+(C_2')\ge \delta - 4$. Consider $C' = w_1\ldots w_k v_{i_1}v_{j_1+2}v_{j_1+3}\dots v_{\delta}v_0v_1\dots v_{i_1-1}v_{j_1+1}v_{i_1+1}\ldots v_{j_1} w_{k+1} \ldots w_{n-\delta-1}w_1$,  which is obtained from $C_1$ by deleting the arc between $w_k$ and $w_{k+1}$, and adding the arc between $w_k$ and $v_{i_1}$, the arc between $w_{k+1}$ and $v_{j_1}$, and the longer path of $C'_2$ from $v_{i_1}$ to $v_{j_1}$. Thus, $\sigma_{\max}(C) \ge \sigma^+(C)\ge n-2\delta +3 - 1 + \delta -4=n-\delta -2$.
     \end{proof}

\section{Proof of Theorem \ref{thm:approx}}\label{sec4}

 The following lemma informs us of a lower bound on the size $a(D)$ of the arc set of $D$. For an undirected graph $G$, let $e(G)=|E(G)|.$
	
	\begin{lemma}\label{lem:1}
		Let $D$ be an oriented graph and $t$ a non-negative integer. If $s^*(D)\geq t$, then $a(D)\geq \frac{n(n+t)}{4}$.
	\end{lemma}
	\begin{proof}
		As $D$ is an oriented graph, we only need to show that the underlying graph $G$ of $D$ has at least $\frac{n(n+t)}{4}$ edges. Consider the complement $\overline{G}$ of $G$. On the one hand, since $s^*(D)\geq t$, we have that
		\begin{eqnarray}\label{eqn:1}
			\sum_{uv\in E(\overline{G})} (d_{\overline{G}}(u)+d_{\overline{G}}(v))&=&\sum_{uv\not\in E(G)}(2n-2-(d_{G}(u)+d_{G}(v)))\nonumber\\
&\leq& e(\overline{G})(n-2-t).
		\end{eqnarray}
		
		On the other hand, by Cauchy-Schwarz inequality, we have that
		\begin{equation}\label{eqn:2}
			\sum_{uv\in E(\overline{G})} (d_{\overline{G}}(u)+d_{\overline{G}}(v))=\sum_{u\in V(\overline{G})} d_{\overline{G}}^2(u)\geq \frac{1}{n}\left(\sum_{u\in V(\overline{G})} d_{\overline{G}}(u)\right)^2=\frac{4e^2(\overline{G})}{n}.
		\end{equation}
		
		By (\ref{eqn:1}) and (\ref{eqn:2}), we have
		\(e(\overline{G})\leq n(n-2-t)/4,\)
		and therefore $e(G)={n\choose 2}-e(\overline{G})\geq \frac{n(n+t)}{4}$, which completes the proof.
	\end{proof}

    We call an undirected graph $G$ a \emph{diamond}, if $G$ is isomorphic to $K_4^-$, where $K_4^-$ is obtained from $K_4$ by removing one edge. Let $V(G)=\{a,b,c,d\}$, $d(a)=d(b)=3$ and $d(c)=d(d)=2$. Given an orientation of $G$, it is a \emph{good diamond} if \QGr{the following two conditions hold: (i) either both $c$ and $d$ have arcs directed towards $a$, or $a$ has arcs directed towards both $c$ and $d$; (ii) either both $c$ and $d$ have arcs directed towards $b$, or $b$ has arcs directed towards both $c$ and $d$} (see Figure \ref{fig1}). Let $D$ be a good diamond. Note that there are oriented paths $P_{cd}$ and $P_{dc}$ of length 3 from $c$ to $d$ and from $d$ to $c$ respectively in $D$, such that $\sigma^+(P_{cd})\ge 2$ and $\sigma^+(P_{dc})\ge 2$.

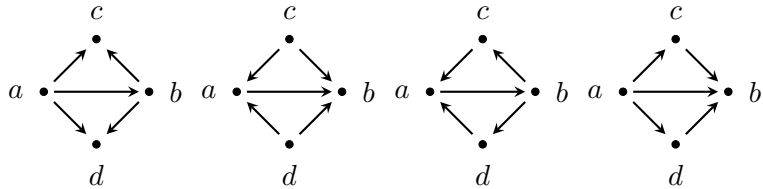
\begin{figure}[htp]
    \centering

    \begin{minipage}[t]{0.19\linewidth}
        \vspace{0pt}
        \centering
        \begin{tikzpicture}[scale=0.7]
            \filldraw[black](0,1) circle (2pt)node[label = left:$a$](a){};
            \filldraw[black](2,1) circle (2pt)node[label = right:$b$](b){};
            \filldraw[black](1,2) circle (2pt)node[label = above:$c$](c){};
            \filldraw[black](1,0) circle (2pt)node[label = below:$d$](d){};

            \foreach \i/\j/\t in {
            a/c/0,
            a/b/0,
            a/d/0,
            b/d/0,
            b/c/0
            }{\path[-stealth, line width = 0.8] (\i) edge[bend left = \t] (\j);}
        \end{tikzpicture}
    \end{minipage}
\begin{minipage}[t]{0.19\linewidth}
        \vspace{0pt}
        \centering
        \begin{tikzpicture}[scale=0.7]
            \filldraw[black](0,1) circle (2pt)node[label = left:$a$](a){};
            \filldraw[black](2,1) circle (2pt)node[label = right:$b$](b){};
            \filldraw[black](1,2) circle (2pt)node[label = above:$c$](c){};
            \filldraw[black](1,0) circle (2pt)node[label = below:$d$](d){};

            \foreach \i/\j/\t in {
            c/a/0,
            a/b/0,
            d/a/0,
            d/b/0,
            c/b/0
            }{\path[-stealth, line width = 0.8] (\i) edge[bend left = \t] (\j);}
        \end{tikzpicture}
    \end{minipage}
    \begin{minipage}[t]{0.19\linewidth}
        \vspace{0pt}
        \centering
        \begin{tikzpicture}[scale=0.7]
            \filldraw[black](0,1) circle (2pt)node[label = left:$a$](a){};
            \filldraw[black](2,1) circle (2pt)node[label = right:$b$](b){};
            \filldraw[black](1,2) circle (2pt)node[label = above:$c$](c){};
            \filldraw[black](1,0) circle (2pt)node[label = below:$d$](d){};

            \foreach \i/\j/\t in {
            c/a/0,
            a/b/0,
            d/a/0,
            b/d/0,
            b/c/0
            }{\path[-stealth, line width = 0.8] (\i) edge[bend left = \t] (\j);}
        \end{tikzpicture}
    \end{minipage}
    \begin{minipage}[t]{0.19\linewidth}
        \vspace{0pt}
        \centering
        \begin{tikzpicture}[scale=0.7]
            \filldraw[black](0,1) circle (2pt)node[label = left:$a$](a){};
            \filldraw[black](2,1) circle (2pt)node[label = right:$b$](b){};
            \filldraw[black](1,2) circle (2pt)node[label = above:$c$](c){};
            \filldraw[black](1,0) circle (2pt)node[label = below:$d$](d){};

            \foreach \i/\j/\t in {
            a/c/0,
            a/b/0,
            a/d/0,
            d/b/0,
            c/b/0
            }{\path[-stealth, line width = 0.8] (\i) edge[bend left = \t] (\j);}
        \end{tikzpicture}
    \end{minipage}

    \caption{The four good diamonds}
    \label{fig1}
\end{figure}
    	
    \QGr{A path forest is a graph in which every component is a path.} The following lemmas can be found in \cite{GKM23}. While the second lemma is a slight generalization of the result by P\'osa~\cite{P63}, it can be shown by slightly modifying his proof for graphs \QGn{$G$} with $\delta(G)\geq (n+k)/2$. However, for the sake of completion, we provide a proof here.
	
	\begin{lemma}[\cite{GKM23}]\label{lem:2}
		Let $k\geq 1$, and let $G$ be a graph with $|V(G)|=n\geq 30+4(k-1)$ and $e(G)\geq \frac{n^2}{4} + 2(k-1)n - 4k^2 +6k-1$. Then, every orientation of $G$ contains $k$ vertex-disjoint good diamonds.
	\end{lemma}

	\begin{lemma}\label{lem:3}
		Let $t\geq 0$ and let $G$ be a graph with $n$ vertices and $s^*(G)\geq t$. Let $E\subseteq E(G)$ be the edge set of a path forest of size at most $t$. Then, there exists a Hamilton cycle in G which uses all edges in $E$.
	\end{lemma}
	\begin{proof}
  Suppose $G$ is a maximal counterexample to this lemma, which means by adding any new edge $uv$, the lemma holds for $G+uv$. Note that $G$ cannot be a complete graph as otherwise this lemma would have held for $G$. Let $v_1$ and $v_n$ be a pair of non-adjacent vertices. As the lemma holds for $G+v_1v_n$ and not for $G$, there is a Hamilton cycle $C=v_1v_2\dots v_nv_1$ where $E\subseteq E(C)$. Let $P$ be the Hamilton path $P=C-\{v_1v_n\}$ in $G$.

For a property $\cal P$, let $I({\cal P})$ be the {\em indicator} of $\cal P$ i.e.  $I({\cal P})=1$ if  $\cal P$ holds and $I({\cal P})=0$, otherwise. For every $i\in [n]$, let $I(v_1v_i)$ and $I(v_nv_i)$ be the indicators for whether $v_1v_i\in E(G)$ and $v_nv_i\in E(G)$ respectively. Let $S=E(P)\setminus E$. Thus, $|S|=n-1-|E|$. Note that for every $v_{i-1}v_{i}\in S$, we have that
    \begin{equation}\label{eq:3}
    	I(v_1v_{i})+I(v_nv_{i-1})\leq 1,
    \end{equation}
    as it is obvious if $i=2$ or $n$ (since $I(v_1v_n)=0$) and if $i\in \{3, 4, \dots, n-1\}$ and $I(v_1v_{i})+I(v_nv_{i-1})= 2$, then $v_1v_iv_{i+1}\dots v_n v_{i-1} v_{i-2}\dots v_1$ is a Hamilton cycle containing all edges in $E$, a contradiction. Thus, on the one hand, by (\ref{eq:3}), we have that
    \begin{eqnarray}\label{eq:4}
    \sum_{v_{i-1}v_i\in E(P)} (I(v_1v_i)+I(v_nv_{i-1})) \leq |S|+2|E|=n+|E|-1\leq n+t-1.
     \end{eqnarray}
     On the other hand,
      \begin{eqnarray*}
   	\sum_{v_{i-1}v_i\in E(P)} (I(v_1v_i)+I(v_nv_{i-1})) &=&\sum_{i=1}^{n-1} I(v_nv_i) +\sum_{i=2}^{n} I(v_1v_i)\\
&=&d(v_n)+d(v_1)\geq n+t,
\end{eqnarray*}
which contradicts (\ref{eq:4}). This completes the proof.
	\end{proof}
	
	Now we give the proof of Theorem \ref{thm:approx}.\\
	
	\QG{\noindent{\bf Theorem \ref{thm:approx}.} \textit{For any integer $k\geq 0$ and oriented graph $D$ with order $n\geq 30+4(k-1)$. If $s^*(D)\geq 8k$, then there is a Hamilton oriented cycle $C$ in $D$ such that $\sigma_{\max}(C)\geq \lceil\frac{n+k}{2}\rceil$.}}	
			
	\begin{proof}
	    Let $D$ be an oriented graph with order $n\ge 30+4(k-1)$ and $s^*(D)\geq 8k$. By Lemma \ref{lem:1}, we have
		
		\[a(D)\geq \frac{n(n+8k)}{4}\geq  \frac{n^2}{4} + 2(k-1)n - 4k^2 +6k-1.\]
Thus, by Lemma \ref{lem:2}, $D$ contains $k$ vertex-disjoint good diamonds. Let $c_i$ and $d_i$ be the two vertices with degree $2$ on diamond $i$, and $P_i$ and $Q_i$ oriented  paths of length 3 from $c_i$ to $d_i$ and $d_i$ to $c_i$ respectively, satisfying $\sigma^+(P_i)\ge 2$ and $\sigma^+(Q_i)\ge 2$. By Lemma \ref{lem:3}, there is a Hamilton oriented cycle $C$ containing all arcs in $P_i$ (for all $i\in [k]$). Without loss of generality, we assume that most of the arcs in $A(C)\setminus (\cup_{i=1}^k A(P_i))$ \QGr{are forward}. Then, we can replace some $P_i$ by $Q_i$ with most of the arcs \QGr{are forward} if necessary. Thus, there are at least $\frac{n-3k}{2}+2k\geq \frac{n+k}{2}$ arcs \QGr{are forward} and therefore $\sigma_{\max}(C)\geq \frac{n+k}{2}$.
 	\end{proof}	
	
\section{Further research of digraph discrepancy}\label{sec:further}

\GGr{Guo et al. \cite{Guo+}  took the  study of digraph discrepancy to a different direction: they investigated the maximum number of forward arcs in Hamilton oriented paths and cycles in digraphs in some classes of generalizations of tournaments. For semicomplete multipartite digraphs and locally semicomplete digraphs, they proved related characterizations leading to polynomial-time algorithms for finding Hamilton oriented paths and cycles with the maximum number of forward arcs. A digraph is {\em semicomplete multipartite} if it can be obtained from a complete multipartite graph by replacing every edge with an arc with the same end-vertices. A digraph is {\em semicomplete} if it is semicomplete multipartite in which there is only one vertex in each colour class. 
A digraph is {\em locally semicomplete}  if the out-neighbourhood and in-neighbourhood of every vertex induce semicomplete digraphs. 
Guo et al. \cite{Guo+}  also observed that for several other classes of digraphs the problems of maximizing the number of forward arcs in Hamilton oriented cycles
are NP-hard since even the problems of determining the existence of Hamilton oriented  cycles in those digraphs are NP-hard. }

\section*{Acknowledgements} JA is partially supported by the National Natural Science Foundation of China (No.12401456, No.12522117) and the Natural Science Foundation of Tianjin (No.24JCQNJC01960). QG is partially supported by China Scholarship Council (No.202406200159).

\end{document}